\makeatletter
	\def\@fnsymbol#1{\ensuremath{\ifcase#1\or \dagger\or \ddagger\or\mathsection\or \mathparagraph\or \|\or **\or \dagger\dagger\or \ddagger\ddagger \else\@ctrerr\fi}}
\makeatother
 
\title{Strong solutions to McKean--Vlasov SDEs with coefficients of Nemytskii-type: the time-dependent case}
\author{Sebastian Grube\thanks{
                  Faculty of Mathematics, Bielefeld University, 33615 Bielefeld, Germany. E-Mail: sgrube@math.uni-bielefeld.de}
        }
\documentclass{article}
\usepackage[a4paper, left=3.55cm, right=3.55cm, top=2cm]{geometry}
\usepackage{setspace}
\usepackage{graphicx}
\usepackage[utf8]{inputenc}
\usepackage{amsmath}
\usepackage{amsfonts}
\usepackage{amssymb}
\usepackage{amsthm}
\usepackage{hyperref}
\usepackage{mathrsfs}
\usepackage{bbm}
\usepackage{esint}
\usepackage{enumerate}
\usepackage{enumitem}

\newcommand{\law}[1]{{\mathcal L}_{#1}}
\newcommand{\LN}{\mathbb N}
\newcommand{\RR}{\mathbb R}

\newcommand{\BBBB}{\mathcal B}

\newcommand{\FF}{\mathcal F}

\newcommand{\PP}{\mathbb P}

\newcommand{\PPPP}{\mathcal P}
\newcommand{\SSS}{\mathscr S}

\newcommand{\inv}{^{-1}}
\newcommand{\rd}{{\RR^d}}

\DeclareMathOperator{\divv}{\mathrm{div}}
\DeclareMathOperator{\esssup}{ess~sup}

\usepackage{nicefrac}

\usepackage{xifthen}

\newcommand{\norm}[2][]{\ifthenelse{\isempty{#1}}
	{\left\lVert#2\right\rVert}
	{\left\lVert#2\right\rVert_{{#1}}}
}

\newtheorem{theorem}{Theorem}[section]

\newtheorem{corollary}[theorem]{Corollary}

\newtheorem{remark}[theorem]{Remark}

\begin{document}
\newpage
\maketitle
%%%%%%%%%%%%%%%%%%%%%%%%%%%
% abstract and keywords
%%%%%%%%%%%%%%%%%%%%%%%%%%%
\begin{abstract}
We consider a large class of nonlinear FPKEs with coefficients of Nemytskii-type depending \textit{explicitly} on time and space, for which it is known that there exists a sufficiently Sobolev-regular Schwartz-distributional solution $u\in L^1\cap L^\infty$. We show that there exists a unique strong solution to the associated McKean--Vlasov SDE with time marginal law densities $u$. In particular, every weak solution of this equation with time marginal law densities $u$ can be written as a functional of the driving Brownian motion.
 Moreover, plugging any Brownian motion into this very functional produces a weak solution with time marginal law densities $u$.
\end{abstract}

\\
\sloppy \textbf{Keywords:} McKean--Vlasov stochastic differential equation, pathwise uniqueness, Yamada--Watanabe theorem, nonlinear Fokker--Planck--Kolmogorov equation

%%%%%%%%%%%%%%%%%%%%%%
% content
%%%%%%%%%%%%%%%%%%%%%%
\section{Introduction}
In this paper we will consider the following McKean--Vlasov stochastic differential equation (abbreviated by McKean--Vlasov SDE or MVSDE) in $\rd$, $d\in\LN$,  with coefficients of Nemytskii-type, which in our case is of the form
\begin{align}\label{MVSDE}
	dX(t) \notag
	=&\ b\left(t,X(t),\frac{d\law{X(t)}}{dx}(X(t))\right)dt  + \sqrt{2a\left(t,X(t),\frac{d\law{X(t)}}{dx}(X(t))\right)}\mathbbm{1}_{d\times d}\ dW(t), \notag \\
	X(0)=&\ \xi,\tag{MVSDE}
\end{align}
where $t\in [0,T]$, $T\in (0,\infty)$, $\mathbbm{1}_{d\times d}$ is the $d$-dimensional unit matrix, $(W_t)_{t\in[0,T]}$ is a standard $d$-dimensional  $(\FF_t)$-Brownian motion and $\xi$ an $\FF_0$-measurable function on some stochastic basis  $(\Omega,\FF,\PP;(\FF_t)_{t\in[0,T]})$, i.e.  a complete, filtered probability space, where $(\FF_t)_{t\in[0,T]}$ is a normal filtration, and $\law{X(t)}:= \PP\circ (X(t))\inv, t\in [0,T]$.
Here, we assume that the drift and diffusion coefficient are given through Borel-measurable functions
\begin{align} \label{coefficients}
	&b : [0,T]\times\rd\times\RR  \to \rd,\\
	&a: [0,T]\times\rd\times\RR \to \RR.
\end{align}
Further, let us define $\beta(\cdot,r):= a(\cdot,r)r, b^*(\cdot,r):=b(\cdot,r)r, r\in\RR$. As in \cite{barbu2021nonlinear}, we impose the following conditions on the coefficients of \eqref{MVSDE}.
\begin{enumerate}[label=(H\arabic*) , wide=0.5em,  leftmargin=*]
	\item \label{condition.a.general} $a \in C^1([0,T]\times\rd\times\RR), a, \partial_r a$ are bounded and, for all $s,t \in [0,T], x\in \rd, r,\bar{r} \in \RR$,
	\begin{align}
		(\beta(t,x,r)-\beta(t,x,\bar{r}))(r-\bar{r}) &\geq \gamma_0 |r-\bar{r}|^2,\label{condition.beta.monotone}\\
		|\partial_r\beta(t,x,r)-\partial_r\beta(s,x,r)|&\leq h(x)|t-s|\partial_r\beta(t,x,r),\\
		|\beta(t,x,r)-\beta(s,x,r)| + |\nabla_x \beta(t,x,r)-\nabla_x\beta(s,x,r)| &\leq h(x)|t-s|(1+|r|), \\
		|\partial_t\beta(t,x,r)|+|\nabla_x \beta(t,x,r)| &\leq h(x)|r|.\label{beta.multiplicativeDerivativeX}
	\end{align}
	\item \label{condition.b.general} For each $(t,x) \in [0,T]\times\rd$, $b(t,x,\cdot) \in C^1(\RR)$, $b, (t,x,r)\mapsto r\partial_r b(t,x,r)$ are bounded and for all $s,t \in [0,T]$, $x\in \rd$, $r \in \RR$,
	\begin{align*}
		|b^*(t,x,r)-b^*(s,x,r)|	\leq h(x)|t-s|(1+|b^*(t,x,r)|),
	\end{align*}
	and
	\begin{align}\label{condition.bstar.2}
		|b^*(t,x,r)|\leq h(x)|r|,
	\end{align}
\end{enumerate}
where $\gamma_0 >0$, $h \in L^\infty(\rd) \cap L^2(\rd), h \geq 0$.
We note that the derivative $\partial_z$ denotes the derivative with respect to a scalar $z$-coordinate. Further, $\divv=\divv_x$, $\nabla = \nabla_x$, $\Delta = \Delta_x$ and $D=D_x$ denote the divergence, gradient, Laplacian and Jacobian with respect to the spacial $x$-coordinate. All the derivatives are supposed to be understood in the sense of Schwartz-distributions.
Here, we would like to point out that condition \eqref{condition.beta.monotone} and the continuity of $a$ in the $r$-variable imply
\begin{align}\label{condition.a.nondegenerate} 
	a\geq\gamma_0> 0,
\end{align}
which means that the diffusion matrix of \eqref{MVSDE} is assumed to be non-degenerate. 

The McKean--Vlasov SDE \eqref{MVSDE} arises from the study of the following type of nonlinear Fokker--Planck--Kolmogorov equation (FPKE)
\begin{align}
	\partial_t u(t,x) + \divv(b(t,x,u(t,x))u(t,x))-\Delta(a(t,x,u(t,x))u(t,x))&=0,\ \ (t,x)\in [0,T]\times\rd, \notag\\
	\left.u\right|_{t=0} &= u_0,\label{PME}\tag{FPKE}
\end{align}
which describes, for example, particle transport in disordered media (see, e.g. \cite{barbu2021nonlinear, barbu2018Prob} and the references therein).
In general, this equation is to be understood in the Schwartz-distributional sense. In this work, we will say that a family $u=(u_t)_{t\in [0,T]}=(u(t,\cdot))_{t\in [0,T]}$ of $L^1(\rd)$-functions is a Schwartz-distributional solution to \eqref{PME} if  
$t \mapsto u_tdx$ is narrowly continuous and 
\begin{align}\label{FPKE.test}
\int_{\RR^d} \varphi(x) u(t,x)dx =& \int_{\RR^d} \varphi(x) u_0(x)dx + \int_0^t \int_{\RR^d} b^i(s,x,u(s,x)) \partial_i \varphi(x) u(s,x)dxds \nonumber\\
 &+ \int_0^t\int_\rd a(s,x,u(s,x))\Delta \varphi(x) u(s,x) dxds,\ \ \forall t\in [0,T],
\end{align}
for each 
$\varphi \in C^{\infty}_c(\RR^d)$ (using Einstein summation convention), where $b=(b^i)_{i=1}^d$. If, additionally, $u_t \in \PPPP_0(\rd)$, i.e. $u_t$ is a probability density on $\rd$, $t \in [0,T]$, then $u$ is simply called a \textit{probability solution} to \eqref{PME}.

Very recently, \eqref{PME} has been investigated 
 in \cite{barbu2021nonlinear} under the conditions \ref{condition.a.general} and \ref{condition.b.general}.
In their work, the authors found that, under their conditions, \eqref{PME} even has a (unique) analytically strong solution $u$ in $H\inv$ with initial condition $u_0 \in D_0$, where $D_0:=\{f \in L^2 : \beta(0,\cdot,f) \in H^1\}$.
If, additionally, $u_0$ is a probability density, then the solution $u$ is a curve of probability densities and is therefore also a probability solution in the above sense. 
This enabled the authors to construct a weak solution to \eqref{MVSDE} via a superposition principle procedure for McKean--Vlasov SDEs presented in \cite[Section 2]{barbu2019nonlinear} (see also \cite[Section 2]{barbu2018Prob}) based upon Trevisan's superposition principle for SDEs (see \cite[Theorem 2.5]{trevisan_super}, generalising \cite{figalli2008}; see also \cite{bogachev2021super} for a recent improvement of both results).
\cite{barbu2021nonlinear} is subsequent to several papers of the same authors in the study of \eqref{PME} with time-homogenous coefficients
(cf. \cite{barbu2018Prob, barbu2019nonlinear, barbu2020solutions, barbu2021evolution} (and references therein)).
Let us note that in \ref{condition.a.general} and \ref{condition.b.general} a number of the assumptions are always fulfilled in the time-homogeneous setting. However, in the time-dependent case these assumptions are owed to the additional technical challenges, which come along with it.
 
We recall that \cite{barbu2021nonlinear} follows the approach developed in \cite{barbu2018Prob} and \cite{barbu2019nonlinear} by first finding a probability solution to \eqref{PME} and then associating a weak solution to \eqref{MVSDE} such that its time marginal law densities are given by such a probability solution. This constitutes a vital part in the realisation of McKean's original idea, proposed in \cite{mckean1966class}, to 
associate Markov processes to certain nonlinear PDEs, covering, in particular, the viscous Burgers' equation, in a way that the process's transition probabilities solve the PDE.

The aim of this paper is to use the above mentioned connection between probability solutions to \eqref{PME} and \textit{weak} solutions to \eqref{MVSDE} in the setup of \cite{barbu2021nonlinear} and show that under the conditions \ref{condition.a.general}, \ref{condition.b.general}, and additionally \ref{condition.b.pathwise} (see p.~\pageref{condition.b.pathwise}) there even exists a (probabilistically) strong solution to \eqref{MVSDE}, which is pathwise unique among all solutions with time marginal law densities $u$, where $u$ is the probability solution to \eqref{PME} with initial condition $u_0 \in \PPPP_0(\rd)\cap D_0$ provided by \cite{barbu2021nonlinear} (for the exact formulation see Corollary \ref{MVSDE.PME.strongExistence} below).  In particular, all weak solutions to \eqref{MVSDE}, whose time marginal law densities coincide with $u$ have the same law in $\PPPP(C([0,T];\rd))$.

In order to achieve this result, we will employ the procedure developed in \cite{grube2021strong}, which builds upon an application of a restricted Yamada--Watanabe theorem for SDEs to McKean--Vlasov SDEs (see \cite[Section 5]{grube2021strong}).

In our case, we consider McKean--Vlasov SDEs with so-called Nemytskii-type coefficients; we regard \eqref{MVSDE} as a McKean--Vlasov SDE of the form
\begin{align*}
	dX(t) \notag
	=&\ F(t,X(t),\law{X(t)})dt  + \sigma(t,X(t),\law{X(t)}) dW(t), \notag \\
	X(0)=&\ \xi,\tag{MVSDE}
\end{align*}
with coefficients
\begin{align*}
	[0,T]\times\rd\times\PPPP(\rd)\ni(t,x,\nu) &\mapsto F(t,x,\nu):= b(t,x,v_a(x)),\\
	[0,T]\times\rd\times\PPPP(\rd)\ni(t,x,\nu) &\mapsto \sigma(t,x,\nu):=\sqrt{2a(t,x,v_a(x))}\mathbbm{1}_{d\times d},
\end{align*}
where $\PPPP(\rd)$ is the set of all Borel probability measures on $\rd$, $v_a$ denotes the version of the density of the absolutely continuous part of $\nu$ both with respect to Lebesgue measure, which is obtained by setting $v_a= 0$ on the complement of its Lebesgue points.
$F$ and $\sigma$ are therefore Borel-measurable functions (cf., e.g. \cite[Remark 3.4]{grube2021strong}).
The dependence of $F$ and $\sigma$ on $\nu$ in terms of $v_a$ evaluated at a fixed point $x$ excludes the continuity of $F$ and $\sigma$ in their measure-component with respect to the topology of weak convergence of probability measures, Wasserstein distance or bounded variation norm.
These types of continuity assumptions are made in the major part of the literature (see, e.g. \cite{delarue2018mckean}).

Let us briefly recall the connection with our previous work.
In the case the coefficients of \eqref{MVSDE} are of the form $b(t,x,r) = E(x)\bar{b}(r)$, $\bar{b}\geq 0$, and $a(t,x,r)= a(r)$, a result analogous to Corollary \ref{MVSDE.PME.strongExistence} was achieved in \cite{grube2021strong} under weaker conditions.\\

This paper is structured as follows.
	First, we will fix some frequently used notation and, afterwards, recall the notation of a $P_{(v_t)}$-solution to and $P_{(v_t)}$-uniqueness for \eqref{MVSDE}, where $v$ is a probability solution to \eqref{PME}.
	Second, in Section~\ref{section.procedure}, we will provide the reader with a two-step procedure on how to obtain a unique strong solution to \eqref{MVSDE} with time marginal law densities $u$, where $u$ is the probability solution to \eqref{PME} with initial condition $u_0 \in \PPPP_0\cap L^\infty\cap D_0$  provided by \cite{barbu2021nonlinear}, based on the procedure developed in \cite{grube2021strong}.
	This procedure will then be carried out in the last section, Section \ref{section.mainresult}, which is divided into three subsections.
	Subsection \ref{section.mainresult.ingredient1} is devoted to gathering the results on the existence and the regularity of $u$ and the existence of a weak solution to \eqref{MVSDE} with time marginal law densities $u$.
	In Subsection \ref{section.mainresult.ingredient2}, we will give a pathwise uniqueness result for \eqref{MVSDE} among all weak solutions with time marginal law densities $u$.
	In Subsection \ref{section.mainresult.yamada}, we will apply the restricted version of the Yamada--Watanabe theorem for SDEs from \cite{grube2021strong} to \eqref{MVSDE}
	and combine the results of Subsections \ref{section.mainresult.ingredient1} and \ref{section.mainresult.ingredient2} in order to obtain a  strong solution to \eqref{MVSDE}, which is pathwise unique given the time marginal law densities $u$. 
\section*{Notation}
Within this paper we will use the following notation, which is essentially taken from \cite{grube2021strong}.
For a topological space $(\textbf{T},\tau)$, $\BBBB(\textbf{T})$ shall denote the Borel $\sigma$-algebra on $(\textbf{T},\tau)$.

Let $n\geq 1$. On $\RR^n$, we will always consider the usual $n$-dimensional Lebesgue measure $\lambda^n$ if not said any differently. If there is no risk for confusion, we will just say that some property for elements in $\RR^n$ holds \textit{almost everywhere} (or $\textit{a.e.}$) if and only if it holds $\lambda^n$-almost everywhere.
Furthermore, on $\RR^n$, $|\cdot|_{\RR^n}$ denotes the usual Hilbert--Schmidt norm. If there is no risk for confusion, we will just write $|\cdot|=|\cdot|_{\RR^n}$. By $B_R(x)$ we will denote the usual open ball with center $x\in \RR^n$ and radius $R>0$. 

Let $(S,\SSS,\eta)$ be a measure space. For $1\leq p \leq \infty$, $L^p(S;E)$ symbolises the usual Bochner space on $S$ with values in $E$. 
If $S=\RR^n$ and $E=\RR$, we just write $L^p(\RR^n;\RR)=L^p(\RR^n)$. The set of locally $p$-integrable functions on $\RR^n$ with values in $E$ will be denoted by $L^p_{loc}(\RR^n;E)$.
Moreover, $W^{1,p}(\RR^n)$ denotes the usual Sobolev-space, containing all $L^p(\RR^n)$-functions, whose first-order distributional derivatives can be represented by elements in $L^p(\RR^n)$. In the case $p=2$, we set $H^1(\RR^n):=W^{1,p}(\RR^n)$ and its continuous dual space shall be denoted by $H\inv(\RR^n)$. Accordingly, vector-valued first-order Sobolev functions on $\RR^n$ will be denoted by $W^{1,p}(\RR^n;E)$.

Let $(M,d)$ be a metric space. Then $\PPPP(M)$ denotes the set of all Borel probability measures on $(M,d)$. We will consider $\PPPP(M)$ as a topological space with respect to the topology of weak convergence of probability measures. 
A curve of probability measures $(\nu_t)_{t\in [0,T]}\subset \PPPP(M)$ is called narrowly continuous if $[0,T]\ni t\mapsto \int \varphi(x)\nu_t(dx)$ is continuous for all $\varphi \in C_b(M)$.
By $\PPPP_0(\RR^n)$ we will denote the set of all probability densities with respect to Lebesgue measure, i.e.
\begin{align*}
	\PPPP_0(\RR^n) = \left\{ \rho \in L^1(\RR^n)\ |\ \rho\geq 0 \text{ a.e.}, \int_{\RR^n} \rho(x) dx =1\right\}.
\end{align*}
Let $(\Omega,\FF,\PP)$ be a probability space and $(S,\SSS)$ a measurable space. If $X:\Omega \to S$ is an $\FF\slash\SSS$-measurable function, then we say that $\law{X}:=\PP\circ X\inv$ is the \textit{law} of $X$, whenever there is no risk for confusion about the underlying probability measure $\PP$.

By $C_c^\infty(\RR^n)$ we denote the set of all infinitely differentiable functions with compact support. 
Let $(E,\norm[E]{\cdot})$ be a Banach space.
The set of continuous functions on the interval $[0,T]$ with values in $E$ is denoted by $C([0,T];E)$ and is considered with respect to the usual supremum's norm.
 Further, we define $$C([0,T];E)_0 := \{w \in C([0,T];E): w(0)=0\}.$$
For $t\in [0,T]$, $\pi_t: C([0,T];E) \to E$ denotes the canonical evaluation map at time $t$, i.e. $\pi_t(w):=w(t), w\in C([0,T];E)$. Further, we set $\BBBB_t(C([0,T];E)):= \sigma(\pi_s : s\in [0,t])$ and, correspondingly, $\BBBB_t(C([0,T];E)_0):= \sigma(\pi_s: s\in [0,t])\cap C([0,T];E)_0$.
Moreover, $\PP_W$ denotes the Wiener measure on $(C([0,T];\rd)_0,\BBBB(C([0,T];\rd)_0))$.

\section*{$P_{(v_t)}$-solutions to \eqref{MVSDE}}
\label{notation.terminology.restricted} 
Let us briefly recall the solution concepts for \eqref{MVSDE} from \cite{grube2021strong} in order to make the steps in Section \ref{section.procedure} and the application of the restricted Yamada--Watanabe theorem in Section \ref{section.mainresult.yamada} conceptually more feasible.

If $v$ is a probability solution to \eqref{PME}, we set
\begin{align*}
	P_{(v_t)}:=\{Q \in \PPPP(C([0,T];\rd)): Q \circ \pi_t\inv = v_tdx, \forall t\in [0,T]\}.
\end{align*}

A \textbf{$P_{(v_t)}$-weak solution} $(X,W,(\Omega,\FF,\PP;(\FF_t)_{t\in[0,T]}))$ is a (probabilistically) weak solution $(X,W, (\Omega,\FF,\PP;(\FF_t)_{t\in[0,T]}))$ in the usual sense such that $\law{X(t)} = v_t$, for all $t\in [0,T]$. For the convenience of the reader, we will just write $(X,W)=(X,W, (\Omega,\FF,\PP;(\FF_t)_{t\in[0,T]}))$ in cases in which do not need to refer explicitly to the underlying stochastic basis $(\Omega,\FF,\PP;(\FF_t)_{t\in[0,T]})$.

We will say that \eqref{MVSDE} has a \textbf{$P_{(v_t)}$-strong solution} if there exists a function $F: \rd \times C([0,T];\rd)_0\to C([0,T];\rd)$, which is $\overline{\BBBB(\rd)\otimes \BBBB(C([0,T];\rd)_0)}^{v_0dx\otimes \PP_W}\slash \BBBB(C([0,T];\rd))$-measurable, such that, for $\mu_0$-a.e. $x\in \rd$, $F(x,\cdot)$ is $\overline{\BBBB_t(C([0,T];\rd)_0)}^{\PP_W}\slash \BBBB_t(C([0,T];\rd))$-measurable for all $t\in [0,T]$ and, whenever $\xi$ is an $\FF_0$-measurable function with $\law{\xi} = v_0dx$ and $W$ is a standard $d$-dimensional $(\FF_t)$-Brownian motion on some stochastic basis $(\Omega,\FF,\PP;(\FF_t)_{t\in [0,T]})$, $(F(\xi,W),W, (\Omega,\FF,\PP;(\FF_t)_{t\in[0,T]}))$ is a $P_{(v_t)}$-weak solution to \eqref{MVSDE}.
Here, $\overline{\BBBB(\rd)\otimes \BBBB(C([0,T];\rd)_0)}^{v_0dx\otimes \PP_W}$ denotes the completion of $\BBBB(\rd)\otimes \BBBB(C([0,T];\rd)_0)$ with respect to the measure $v_0dx\otimes \PP_W$, and $\overline{\BBBB_t(C([0,T];\rd)_0)}^{\PP_W}$ denotes the completion of $\BBBB_t(C([0,T];\rd)_0)$ with respect to $\PP_W$ on $(C([0,T];\rd)_0,\BBBB(C([0,T];\rd)_0))$.

Moreover, \textbf{$P_{(v_t)}$-pathwise uniqueness} holds for \eqref{MVSDE} if for every two $P_{(v_t)}$-weak solutions $(X,W, (\Omega,\FF,\PP;(\FF_t)_{t\in[0,T]}))$, $(Y,W, (\Omega,\FF,\PP;(\FF_t)_{t\in[0,T]}))$ (with respect to the same Brownian motion on the same stochastic basis) with $X(0)=Y(0)$ $\PP$-a.s., one has \\${\sup_{t\in [0,T]}|X(t) - Y(t)|}$ $\PP$-a.s.

We say that there exists a \textbf{unique $P_{(v_t)}$-strong solution} to \eqref{MVSDE} if there exists a $P_{(v_t)}$-strong solution to \eqref{MVSDE} with functional $F$ as above, and every $P_{(v_t)}$-weak solution $(X,W)$ is of the form $X = F(X(0),W)$ almost surely with respect to the underlying probability measure.
\section{The procedure}\label{section.procedure}
\label{chapter.PU.Procedure}
We will essentially follow the same procedure as proposed in \cite{grube2021strong}.

Our overall goal is to apply a restricted Yamada--Watanabe theorem for SDEs to \eqref{MVSDE} (see Theorem \ref{MVSDE.restrYamada} below), which will enable us to show that there exists a unique $P_{(u_t)}$-strong solution to \eqref{MVSDE} under the conditions \ref{condition.a.general}, \ref{condition.b.general} and \ref{condition.b.pathwise} (see below), where $u$ is the analytically strong solution in $H\inv$ to \eqref{PME} with initial condition $u_0\in \PPPP_0 \cap D_0\cap L^\infty$ provided by \cite{barbu2021nonlinear} (see Theorem \ref{PME.existence} below). In order to achieve this, we will need the following ingredients.
\begin{enumerate}
	\item \label{procedure.1} A $P_{(u_t)}$-weak solution to \eqref{MVSDE},
	\item \label{procedure.2} $P_{(u_t)}$-pathwise uniqueness holds for \eqref{MVSDE}. In order to show this, we will use that ${u \in L^2([0,T];H^1(\rd))\cap L^\infty([0,T]\times \rd)}$ (see Theorem \ref{PME.existence}).
\end{enumerate}
The ingredients will be gathered in Sections \ref{section.mainresult.ingredient1} and \ref{section.mainresult.ingredient2} and combined via the previously mentioned restricted Yamada--Watanabe theorem in Section \ref{section.mainresult.yamada}.
\section{The main result}\label{section.mainresult}
\subsection{Ingredient 1: A $P_{(u_t)}$-weak solution to \eqref{MVSDE}}\label{section.mainresult.ingredient1}
In \cite{barbu2021nonlinear}, Barbu and Röckner studied \eqref{PME} as an evolution equation in $H\inv$ in the analytically strong sense.
The following existence result for a Schwartz-distributional solution to \eqref{PME} is a special case of \cite[Theorem  2.1]{barbu2021nonlinear} and we will, thus, omit the proof.
\begin{theorem}\label{PME.existence}
	 Let $D_0:= \{v \in L^2(\rd): \beta(0,\cdot,v) \in H^1(\rd)\}$. Assume that \ref{condition.a.general} and \ref{condition.b.general} hold.	Then, for each $u_0 \in L^1(\rd)\cap D_0$ there exists a Schwartz-distributional solution $u=(u_t)_{t\in [0,T]}$ to \eqref{PME} such that
	\begin{align}\label{PME.existence.regularity}
		u &\in C([0,T];L^2(\rd))\cap W^{1,2}([0,T]; H\inv(\rd)) \cap L^\infty([0,T];L^1(\rd)),\\
		u, \beta(\cdot,u) &\in L^2([0,T];H^1(\rd)),\notag
	\end{align}
	and $\norm[L^1(\rd)]{u(t)} \leq \norm[L^1(\rd)]{u_0}$ for all $t\in [0,T]$.
	If $u_0 \in \PPPP_0(\rd) \cap D_0$, then $u$ is a probability solution to \eqref{PME}.
	
	Finally, assume that the condition \eqref{condition.bstar.2} is replaced by the stronger condition
	\begin{align}\label{PME.existence.regularity.strongerConditionB}
		|b^*(t,x,r)-b^*(t,x,\bar{r})|\leq h(x)|r-\bar{r}|,\ \ \forall t\in [0,T], x\in \rd,\ r,\bar{r} \in \RR,
	\end{align}
	where $h \in (L^2\cap L^\infty)(\rd), h\geq 0$ is the function introduced in \ref{condition.a.general} and \ref{condition.b.general}. 
	Then, for all $u_0, \bar{u}_0 \in L^1(\rd)\cap D_0$, the corresponding solutions $u(t,u_0), u(t,\bar{u}_0)$ to \eqref{PME} satisfy
	\begin{align}\label{PME.existence.L1contraction}
		\norm[L^1(\rd)]{u(t,u_0)- u(t,\bar{u}_0)} \leq \norm[L^1(\rd)]{u_0-\bar{u}_0},\ \ \forall t\in [0,T].
	\end{align}
	Moreover, if $D_x b\in L^1_{loc}([0,T]\times\rd\times\RR;\RR^{d\times d}),\Delta_x \beta \in L^1_{loc}([0,T]\times\rd\times\RR;\rd)$, %$t \in [0,T],r\in \RR$ 
	and
	\begin{align*}
		\Lambda(b,\beta):=\esssup\{|D_x b(t,x,r) r| + |\Delta_x\beta(t,x,r)| : t\in [0,T], x\in \rd, r\in\RR\}<\infty,	
	\end{align*}
	then, for all $u_0 \in L^1(\rd)\cap D_0\cap L^\infty(\rd)$,
$u \in L^\infty([0,T]\times\rd)$ with 
\begin{align*}
	\norm[{L^\infty([0,T]\times\rd)}]{u} \leq \Lambda(b,\beta)T + \norm[L^\infty(\rd)]{u_0}.
\end{align*}
	\end{theorem}
The following theorem can be found in the same work \cite[Corollary  2.3]{barbu2021nonlinear}, which is based on the superposition principle procedure from \cite[Section 2]{barbu2018Prob}.
\begin{theorem}[$P_{(u_t)}$-weak solution]\label{MVSDE.existence}
	Assume \ref{condition.a.general} and \ref{condition.b.general} and  $u_0 \in \PPPP_0(\rd)\cap D_0$. Then, there exists a $P_{(u_t)}$-weak solution $(X,W)$ to \eqref{MVSDE}
	where $u$ is the probability solution to \eqref{PME} provided by Theorem \ref{PME.existence}.
\end{theorem}

\subsection{Ingredient 2: $P_{(u_t)}$-pathwise uniqueness for \eqref{MVSDE}}\label{section.mainresult.ingredient2}
Let $u$ be the probability solution to \eqref{PME} with initial condition $u_0 \in \PPPP_0(\rd)\cap D_0$ provided by Theorem \ref{PME.existence}.
Recall that, by definition, $P_{(u_t)}$-weak solutions to \eqref{MVSDE} all have the same time marginal laws. In particular, all these solutions fulfill the following SDE(!)
\begin{align}\label{MVSDE.fixed.u}\tag{$\text{SDE}_u$}
	dX(t)&=b^u(t,X(t))dt + \sqrt{2a^u(t,X(t))}\mathbbm{1}_{d\times d}dW(t),\ \ t\in [0,T],\\
	\law{X(0)} &= u_0(x)dx,\notag
\end{align}
where $b^u(t,x):=b(t,x,u_t(x))$ and $a^u(t,x):=a(t,x,u_t(x))$, $(t,x)\in [0,T]\times\rd$, and where, for each $t\in [0,T]$, we consider the $\lambda^d$-version of $u_t$ obtained by setting $u_t = 0$ on the complement of its Lebesgue points (cf. Introduction and \cite[Remark 2.8]{grube2021strong}).

Our aim is to show $P_{(u_t)}$-pathwise uniqueness for \eqref{MVSDE.fixed.u}, which is obviously equivalent to show $P_{(u_t)}$-pathwise uniqueness for \eqref{MVSDE}.
We will do so via a pathwise uniqueness result for SDEs extracted from the proof of \cite[Theorem 1.1]{roeckner2010weakuniqueness}. In order to be able to apply it, we need some additional assumptions. In particular, we need to guarantee that $u$ is bounded. These assumptions are formulated in the following.
\begin{enumerate}[label=(H\arabic*) , wide=0.5em,  leftmargin=*]\setcounter{enumi}{2}
\item \label{condition.b.pathwise} Assume that \eqref{PME.existence.regularity.strongerConditionB} holds, $\Lambda(b,\beta) < \infty$ 
 (see Theorem \ref{PME.existence} for its definition) and $b\in C^1([0,T]\times\rd\times\RR)$.
\end{enumerate}
The following theorem provides the main result of this subsection.
\begin{theorem}[$P_{(u_t)}$-pathwise uniqueness]\label{MVSDE.PU}
Assume that \ref{condition.a.general}, \ref{condition.b.general} and \ref{condition.b.pathwise} hold.
Let $u_0 \in \PPPP_0(\rd)\cap  D_0\cap L^\infty(\rd)$ and let $u$ denote the corresponding probability solution to \eqref{PME} provided by Theorem \ref{PME.existence}. 
Let $(X,W), (Y,W)$ be two $P_{(u_t)}$-weak solutions on a common stochastic basis $(\Omega,\FF,\PP;(\FF_t)_{t\in [0,T]})$ with respect to the same standard $d$-dimensional  $(\FF_t)$-Brownian motion $W$.
	
Then, $\sup_{t\in [0,T]} |X(t)-Y(t)| = 0\  \PP\text{-a.s.}$
\end{theorem}
\begin{proof}
As already mentioned before, $(X,W)$ and $(Y,W)$ both solve \eqref{MVSDE.fixed.u} and we will show $P_{(u_t)}$-pathwise uniqueness for solutions to \eqref{MVSDE.fixed.u}. Therefore, we will check the assumptions from \cite[Theorem 2.4]{grube2021strong} (extracted from the proof of \cite[Theorem 2.5]{roeckner2010weakuniqueness}). 
These assumptions are implied by the following conditions:
\begin{enumerate}[label=(\roman*)] 
	\item $b^u \in L^\infty([0,T]\times\rd;\rd), a^u \in L^\infty([0,T]\times\rd)$,
	\item $b^u$ and $\sqrt{a^u}$ are weakly differentiable in the $x$-coordinate such that
	\begin{align*}
		D_x b^u \in L^2([0,T];L^2_{loc}(\rd;\RR^{d\times d})), \nabla_x \sqrt{a^u} \in L^2([0,T];L^2_{loc}(\rd;\rd)).
	\end{align*}
\end{enumerate}
Clearly, (i) is satisfied. Let us now consider (ii).
Recall that Theorem \ref{PME.existence} guarantees that $u\in L^2([0,T];H^1(\rd))\cap L^\infty([0,T]\times\rd)$. Hence, by the usual chain-rule for Sobolev functions, for a.e. $t\in [0,T]$, $b^u(t,\cdot)$ has a weak gradient, with the representation
\begin{align*}
	D_x b^u(t,x) = (D_x b)(t,x,u(t,x))+(\partial_r b)(t,x,u(t,x))\nabla_x^T u(t,x)
\end{align*}
for a.e. $x \in \rd$.
Let $R>0$. Due to the local boundedness of the derivatives of $b$, we may find a constant $C>0$ such that
\begin{align*}
	\int_0^T\int_{B_R(0)} |D_{x}b^u(t,x)|^2 dxdt\leq TC|B_R(0)| + C \int_0^T\int_{B_R(0)} |\nabla_x u(t,x)|^2dxdt
	< \infty.
\end{align*}
Similarly, by \ref{condition.a.general},
 $a^u(t,\cdot)$ has a weak gradient with
\begin{align*}
	\nabla_x a^u(t,x)=(\nabla_x a)(t,x,u(t,x))+(\partial_{r}a)(t,x,u(t,x))\nabla^T_{x}u(t,x),
\end{align*}
for almost every $x\in \rd$.
It is easy to extend the restricted square-root function $\sqrt{\ \cdot\ }_{|[\gamma_0, \infty)}$ to a function $h_{\sqrt{\ }}\in C^1(\RR)$ with $h_{\sqrt{\ }}(x) = \sqrt{x}, x\in [\gamma_0,\infty)$.
Now, employing \eqref{condition.a.nondegenerate}, we may calculate with the help of the usual chain-rule for Sobolev-functions
\begin{align*}
	\nabla_x \sqrt{a^u(t,x)}
	&=\nabla_x h_{\sqrt{\ }}(a^u(t,x))
	=h_{\sqrt{\ }}'(a^u(t,x))\nabla_x(a^u(t,x)) \\
	&=\frac{(\nabla_x a)(t,x,u(t,x))+(\partial_{r}a)(t,x,u(t,x))\nabla_{x}u(t,x)}{2\sqrt{a(t,x,u(t,x))}},
\end{align*}
for almost every $x\in \RR^n$.
Since $a \in C^1([0,T]\times\rd\times\RR)$ and $ u\in L^2([0,T];H^1(\rd))\cap L^\infty([0,T]\times\rd)$, for each $R>0$, we can find a constant $C>0$ such that
\begin{align*}
	\int_0^T\int_{B_R(0)}|\nabla_{x}\sqrt{a^u(t,x)}|^2dxdt \leq \frac{C}{2\sqrt{\gamma_0}}\left(T|B_R(0)| + \int_0^T\int|\nabla_x u(t,x)|^2dxdt<\infty
\right)<\infty.
\end{align*}
This concludes the proof.
\end{proof}
\subsection{Application of the restricted Yamada--Watanabe theorem to \eqref{MVSDE}}
\label{section.mainresult.yamada}
Now let us combine the ingredients from the previous two subsections and apply the \textit{restricted} Yamada--Watanabe theorem for SDEs obtained in \cite{grube2021strong} to \eqref{MVSDE} in order to obtain the unique $P_{(u_t)}$-strong solution to \eqref{MVSDE}. For the terminology of $P_{(u_t)}$-solutions and uniqueness, please consult page \pageref{notation.terminology.restricted}. 

The following theorem is a special case of \cite[Theorem 3.3]{grube2021strong}.

\begin{theorem}
\label{MVSDE.restrYamada}
Let $v=(v_t)_{t\in [0,T]}$ be a probability solution to \eqref{PME}.
The following statements regarding \eqref{MVSDE} are equivalent. 
\begin{enumerate}
	\item There exists a $P_{(v_t)}$-weak solution and $P_{(v_t)}$-pathwise uniqueness holds.
	\item There exists a unique $P_{(v_t)}$-strong solution to \eqref{MVSDE}.
\end{enumerate}
\end{theorem}
Gathering the results of Section 3.1 and 3.2, we have the corollary, which represents the main result of this paper.
\begin{corollary}[unique $P_{(u_t)}$-strong solution]\label{MVSDE.PME.strongExistence}
Assume that \ref{condition.a.general}, \ref{condition.b.general} and \ref{condition.b.pathwise} hold. Let $u_0 \in \PPPP_0(\rd)\cap D_0\cap L^\infty(\rd)$ and let $u$ denote the probability solution to \eqref{PME} provided by Theorem \ref{PME.existence}.
Then, there exists a unique $P_{(u_t)}$-strong solution to \eqref{MVSDE}.
\end{corollary}
\begin{remark}
In the situation of Corollary \ref{MVSDE.PME.strongExistence}, the condition on the initial datum can be relaxed to $u_0 \in \PPPP_0(\rd)\cap L^\infty(\rd)$. 
The argumentation sketches as follows.
Assuming the condition \eqref{PME.existence.regularity.strongerConditionB} and using the $L^1$-contraction property \eqref{PME.existence.L1contraction} of the solution $u$, Barbu and Röckner argued that, by an approximation argument, for $u_0 \in \PPPP_0(\rd)$ there exists a probability solution $u$ to \eqref{PME} (cf. \cite[Remark 2.2, Remark 3.3]{barbu2021nonlinear}), which, by the superposition principle, can be lifted to a $P_{(u_t)}$-weak solution to \eqref{MVSDE}.
Additionally assuming that $\Lambda(b,\beta)<\infty$, in fact, the same approximation argument also yields that $u\in L^\infty([0,T]\times\rd)$ and, using \cite[(3.26)]{barbu2021nonlinear}, $u \in L^2([0,T];H^1(\rd))$. Consequently, the argumentation in the proof of Theorem \ref{MVSDE.PU} works also if the assumption $u_0 \in D_0$ is dropped and, thus, $P_{(u_t)}$-pathwise uniqueness holds under this relaxed condition on $u_0$.
	 
In addition, the condition $b \in C^1([0,T]\times\rd\times\RR)$ in \ref{condition.b.pathwise} can be considerably relaxed to the following one: 
		Assume that for every compact set $K\subset \rd \times\RR$ there exists a nonnegative function $g \in L^2([0,T])$ such that for all $t \in [0,T]$ and all $z,\bar{z} \in K$,
\begin{align*}
			|b(t,z)-b(t,\bar{z})| \leq g(t) |z-\bar{z}|.
			\end{align*} 
Then, using a chain-rule for compositions of a Lipschitz function with a Sobolev function as, for example, in \cite[Corollary 3.2]{ambrosio1990chainrule}, Theorem \ref{MVSDE.PU} can be proved in a similar way. For the details, please consult \cite{grube2022thesis}. 
\end{remark}
%%%%%%%%%%%%%%%%%%%%%%%%%%%%%%%%
% Acknowledgements
%%%%%%%%%%%%%%%%%%%%%%%%%%%%%%%%
\textbf{Acknowledgements:}
I am very grateful to Prof. Dr. Michael Röckner, who pointed out this interesting topic to me and with whom I shared various fruitful discussions.
Moreover, I gratefully acknowledge the support by the German Research Foundation (DFG) through the IRTG 2235.
\bibliographystyle{alpha}
\bibliography{Bibliography}
\end{document}